\documentclass[12pt,reqno]{amsart}
\usepackage[latin1]{inputenc}
\usepackage[T1]{fontenc}
\usepackage{mathptmx}
\usepackage{amssymb}
\usepackage[colorlinks=true]{hyperref}
\usepackage{cleveref}

\usepackage{tikz}
\usetikzlibrary{calc}

\usepackage{color}

\newtheorem{theorem}{Theorem}[section]
\newtheorem{lemma}[theorem]{Lemma}
\newtheorem{corollary}[theorem]{Corollary}
\newtheorem{proposition}[theorem]{Proposition}
\newtheorem{conjecture}[theorem]{Conjecture}

\theoremstyle{definition}
\newtheorem{remark}[theorem]{Remark}
\newtheorem{definition}[theorem]{Definition}
\newtheorem{notation}[theorem]{Notation}

\numberwithin{equation}{section}

\def\ZZ{{\mathbb Z}}
\def\NN{{\mathbb N}}

\newcommand{\KK}{\mathbb{K}}

\newcommand{\Ann}{\mathrm{Ann}}
\def\sdepth{\operatorname{sdepth}}
\def\hdepth{\operatorname{hdepth}}
\def\depth{\operatorname{depth}}

\newcommand{\set}[1]{\{#1\}}
\newcommand{\with}{\,:\,}

\newcommand{\pij}{P^g_{I/J}}
\newcommand{\pijp}{P^{g'}_{I'/J'}}

\newcommand{\subsetIJ}{\subsetneq}

\begin{document}
\title[The behavior of Stanley depth under polarization]{The behavior of Stanley depth under polarization}

\author{Bogdan Ichim}

\address{Simion Stoilow Institute of Mathematics of the Romanian Academy, Research Unit 5, C.P. 1-764,
014700 Bucharest, Romania} \email{bogdan.ichim@imar.ro}

\author{Lukas Katth\"an}

\address{Universit\"at Osnabr\"uck, FB Mathematik/Informatik, 49069
Osnabr\"uck, Germany}\email{lukas.katthaen@uos.de}

\author{Julio Jos\'e Moyano-Fern\'andez}

\address{Universidad Jaume I, Campus de Riu Sec, Departamento de Matem\'aticas, 12071
Caste\-ll\'on de la Plana, Spain} \email{moyano@uji.es}

\subjclass[2010]{Primary: 05E40; Secondary: 16W50.}
\keywords{Monomial ideal; Stanley depth; Stanley decomposition; poset map; polarization.}
\thanks{The first author was partially supported  by project  PN-II-RU-TE-2012-3-0161, granted by the Romanian National Authority for Scientific Research,
CNCS - UEFISCDI. The second author was partially supported by
the German Research Council DFG-GRK~1916. The third author was partially supported by the Spanish Government Ministerio de Econom\'ia y Competitividad (MEC), grant MTM2012--36917--C03--03 in cooperation with the European Union in the framework of the founds ``FEDER'', as well as by the DFG-GRK~1916.}

\begin{abstract}
Let $\KK$ be a field, $R=\KK[X_1, \ldots , X_n]$ be the polynomial ring and $J \subsetIJ I$ two monomial ideals in $R$.
In this paper we show that
\[\sdepth\ {I/J} - \depth\ {I/J} = \sdepth\ {I^p/J^p}-\depth\ {I^p/J^p}, \]
where $\sdepth I/J$ denotes the Stanley depth and $I^p$ denotes the polarization.
This solves a conjecture by Herzog \cite{H} and reduces the famous Stanley conjecture (for modules of the form $I/J$) to the squarefree case.
As a consequence, the Stanley conjecture for algebras of the form $R/I$ and the well-known combinatorial conjecture that every Cohen-Macaulay simplicial complex is partitionable are equivalent.
\end{abstract}

\maketitle

\section{Introduction} \label{section:intro}

In 1982, R. Stanley conjectured in his celebrated paper \cite{St} an upper bound for the depth of a multigraded module of combinatorial nature, called \emph{Stanley depth} later on.
A proof of this conjecture turned out to be a difficult problem: it began soon to be called the \emph{Stanley conjecture}.
Since then, several authors began to study intensively this problem, starting with
the reformulation by Apel of the most important cases of the conjecture, i.e. the Stanley conjecture for a monomial ideal $I$ and for the factor ring $R/I$, see \cite[Conjecture 2]{A1} and \cite[Conjecture 1]{A2}.
Afterwards,
most of the research concentrates on the particular case of a module
of the form $I/J$ for two monomial ideals $J \subsetIJ I$ in the
polynomial ring $R=\KK[X_1, \ldots , X_n]$ over some field $\KK$; motivated by works of Herzog and Popescu \cite{HP,P1}, the Stanley conjecture became one important open problem in algebra and combinatorics.
\medskip

A natural first step to approach the Stanley conjecture is to try to reduce it to squarefree monomial ideals.
The arguable most straightforward method for this is via polarization.
This is a process which replaces an arbitrary monomial ideal $I$ with a certain squarefree monomial ideal $I^p$, such that $I$ can be recovered from $I^p$ by dividing out a regular sequence.
The behavior of many invariants of $I$ under polarization is well understood.
In particular, as polarization preserves the projective dimension, the change in the depth is just the change in the number of variables.
In view of the Stanley conjecture, one would hope for a similar behavior of the Stanley depth.
This was formulated as a conjecture by Herzog in the survey \cite{H} as follows:
\begin{conjecture}[Conjecture 62, \cite{H}]\label{conj:HH}
Let $I \subset R$ be a monomial ideal. Then
\[ \sdepth R/I - \depth R/I = \sdepth R^p/I^p - \depth R^p / I^p \]
where $R^p$ is the ring where $I^p$ is defined.
\end{conjecture}

However, despite the naturality of the question and a considerable effort, it remained open for quite some time.
The main result of our paper (\Cref{cor:conj}) is the proof of Conjecture \ref{conj:HH}.
We show it even more generally for modules of the form $I/J$ for two monomial ideals $J \subsetIJ I \subseteq R$.
\medskip

This has two important consequences.
First, it immediately follows that $I/J$ satisfies the Stanley conjecture if and only if its polarization $I^p/J^p$ does so, cf. Corollary \ref{cor:polarization}.
Thus the Stanley conjecture for modules of the form $I/J$---in particular \cite[Conjecture~2]{A1} and \cite[Conjecture~1]{A2}---is effectively reduced to the squarefree case.

For the second consequence, as noted by Stanley himself in \cite[p.~191]{St}, the Stanley conjecture was formulated such that "the question raised in \cite[p.~149,~line 6]{St3} or \cite[Rmk.~5.2]{G} would follow affirmatively". This question was reformulated by Stanley \cite[Conjecture 2.7]{St2}, and asks whether 
every Cohen-Macaulay simplicial complex is partitionable. While it is clear that the Stanley conjecture implies the Garsia-Stanley conjecture on Cohen-Macaulay simplicial complexes, in this paper we show that the converse is also true, that is, the Garsia-Stanley conjecture on Cohen-Macaulay simplicial complexes implies what is (arguably) the most important case of the Stanley conjecture. More precisely, we show that Stanley's conjecture on Cohen-Macaulay simplicial complexes is equivalent to \cite[Conjecture~1]{A2} (see \Cref{cor:equ}).
\medskip

The content of the paper is organized as follows. Section \ref{section:pre} is devoted to explain the prerequisites and to fix notations.
Section \ref{section:posetmaps} introduces a suitable tool for the measure of the Stanley depth of a quotient $I/J$ as above, namely the \emph{maps changing the Stanley depth}. Helpful properties of those maps are recorded in both Lemma \ref{lem:interval} and Lemma \ref{lem:1dim}.

The most remarkable application of Section \ref{section:posetmaps} is that to the polarization of the quotient $I/J$; this is exactly the content of Section \ref{section:polarizations}, which also includes the main results of the paper---the already mentioned 
Corollaries \ref{cor:conj}, \ref{cor:polarization} and \ref{cor:equ}.

Further applications of the maps changing Stanley depth are described in Section \ref{section:furtherapplications}, closing the paper.
\medskip

The reader is referred to Bruns and Herzog \cite{BH}, and Miller and Sturmfels \cite{millersturm} for general definitions, notation, and background material.

\section{Prerequisites}\label{section:pre}
Let $\KK$ be a field. We consider the polynomial ring $R=\KK[X_1,\ldots , X_n]$ over $\KK$, endowed with the fine $\ZZ^n$-grading (i.e. the $\mathbb{Z}^n$-grading with $\deg X_i=e_i$ being the $i$-th vector of the canonical basis).

Let $M$ be a finitely generated graded $R$-module, and $m \in M$ homogeneous. Let $Z \subset \{X_1, \ldots , X_n\}$ be a subset of the set of indeterminates of $R$. The $\KK[Z]$-submodule $m\KK[Z]$ of $M$ is called a \emph{Stanley space} of $M$ if $m\KK[Z]$ is free (as $\KK[Z]$-submodule).
A \emph{Stanley decomposition} of $M$ is a finite family
\[
\mathcal{D}=(\KK[Z_i],m_i)_{i\in \Omega}
\]
in which $Z_i \subset \{X_1, \ldots ,X_n\}$ and $m_i\KK[Z_i]$ is a Stanley space of $M$ for each $i \in \Omega$ with
\[
M = \bigoplus_{i \in \Omega} m_i\KK[Z_i]
\]
as a graded or multigraded $\KK$-vector space. This direct sum carries the structure of $R$-module and has therefore a well-defined depth. The \emph{Stanley depth} $\sdepth\ M$ of $M$ is defined to be the maximal depth of a Stanley decomposition of $M$. The Stanley conjecture states the following inequality:

\begin{conjecture}[Stanley] $\sdepth\ M \geq \depth\ M$. \end{conjecture}

For a recent account on the subject, the reader is referred to Herzog's survey \cite{H}.
\medskip

We endow $\ZZ^n$ (or $\NN^n$) with the componentwise (partial) order: Given $a,b\in \ZZ^n$, we say that $a \leq b$ if and only if $a_i\le b_i$ for $i=1,\ldots,n$. Note that this partial order turns $\ZZ^n$ into a distributive lattice with meet $a\wedge b$ and join $a\vee b$ being the componentwise minimum and maximum, respectively. For $a,b\in \ZZ^n$  the \emph{interval between $a$ and $b$} is defined to be
\[
[a,b]:=\{c \in \ZZ^n \ |\  a \leq c \leq b \}.
\]
Remark here that for $n \in \NN$ we will use the notation $[n]:=\{1,\ldots ,n\}$.
\medskip

Monotonic poset maps will play a prominent role in Section \ref{section:posetmaps}. Let $P\subset\ZZ^n,\  P'\subset\ZZ^{n'}$ be posets. We call a map $\phi: P \to P'$ \emph{monotonic} if it preserves the order. Moreover, a $\phi$ is said to \emph{preserve joins} resp. \emph{meets} if it satisfies $\phi(a \vee b)=\phi(a) \vee \phi(b)$ resp. $\phi(a \wedge b)=\phi(a) \wedge \phi(b)$ for all $a, b \in P$.
\medskip

For $a=(a_1,\ldots ,a_n)\in \NN^n$ we denote by $X^a$ the monomial $X_1^{a_1}\cdots X_n^{a_n}$. Let $J \subsetIJ I \subset R$ be two monomial ideals. The quotient $I/J$ is a graded or multigraded $R$-module. Following Herzog, Vladoiu and Zheng \cite{HVZ}, we fix a vector $g \in \NN^n$ satisfying $a \leq g$ for all $X^a$ in minimal sets of generators for $I$ and $J$.
 The \emph{characteristic poset}
$P^g_{I/J}$ of $I/J$ with respect to $g$ is defined to be the (finite) subposet
\[
P^g_{I/J} :=\{a \in \NN^n : X^a \in I \setminus J, \ \ a \leq g\}
\]
of $\ZZ^n$.
A \emph{partition} of a finite poset $P$ is a disjoint union
\[
\mathcal{P}:\ P=\bigcup_{i=1}^{r} [a^i,b^i]
\]
of intervals. A key result in \cite{HVZ} describes a way to compute $\sdepth\ {I/J}$ from a Stanley decomposition of $I/J$ coming from a partition of the poset $P^g_{I/J}$. More precisely, by setting $Z_b:=\{X_j : b_j=g_j\}$ for each $b \in P^g_{I/J}$, and the function $\rho=\rho^g: P^g_{I/J} \rightarrow \ZZ_{\geq 0},~ \rho (c)=\# (Z_{c})$, Theorem 2.1 in \cite{HVZ} says:

\begin{theorem}[Herzog, Vladoiu, Zheng]\label{theo:HVZ}\hfill

\noindent (a) Let $\mathcal{P}:\ P^g_{I/J}=\bigcup_{i=1}^{r} [a^i,b^i]$ be a partition of $P^g_{I/J}$, then
\[
\mathcal{D}(\mathcal{P}): \ I / J = \bigoplus_{i=1}^r \left ( \bigoplus_c X^c \KK[Z_{b^i}] \right )
\]
is a Stanley decomposition of $I/J$, where the inner direct sum is taken over all $c \in [a^i,b^i]$ for which $c_j=a^i_j$ for all $j$ with $X_j \in Z_{b^i}$. Moreover, we have
\[
\sdepth\ {\mathcal{D}(\mathcal{P})}=\min \{\rho(b^i) : i = 1, \ldots , r\}.
\]
(b) Let $\mathcal{D}$ be a Stanley decomposition of $I/J$. Then there exists a partition $\mathcal{P}$ of $P^g_{I/J}$ such that $\sdepth\ {\mathcal{D}(\mathcal{P})}\geq \sdepth\ {\mathcal{D}}$. In particular, $\sdepth\ {I/J}$ can be computed as the maximum of the numbers $\sdepth\ {\mathcal{D}(\mathcal{P})}$, where $\mathcal{P}$ runs over the (finitely many) partitions of $P^g_{I/J}$.
\end{theorem}
\medskip

The main topic of this paper is the behavior of Stanley depth under polarization.
We recall the definition following Herzog and Hibi \cite{HH}.
Let $I \subset R$ be a monomial ideal with generators $u_1, \ldots, u_m$, where $u_i=\prod_{j=1}^{n} X_j^{a_{ij}}$ for $i=1, \ldots , m$. For each $j$ let $a_j=\max \set{a_{ij} : i=1, \ldots , m}$.
Set $a=(a_1,\ldots,a_n)$ and choose a $g \in \NN^n$ such that $a \leq g$. Set $R^p$ to be the polynomial ring
\[
R^p:=\KK[X_{jk} \with 1 \leq j \leq n, 1 \leq k \leq g_j].
\]

Then the \emph{polarization of $I$} is the squarefree monomial ideal $I^p \subset R^p$ generated by $v_1, \ldots , v_m$, where
\[
v_i=\prod_{j=1}^{n} \prod_{k=1}^{a_{ij}} X_{jk} \ \ \ \mbox{for} \ \ i=1, \ldots, m.
\]
\medskip

\section{Poset maps} \label{section:posetmaps}

Let $I,J$ be monomial ideals of $R$ such that $J \subsetIJ I$.
We are interested in measuring the Stanley depth of the deformations of the quotient $I/J$. The following kind of maps reveals to be a useful tool for that aim:

\begin{definition} \label{def:1}
Let $\ell \in \ZZ$ and $n, n' \in \NN$. A monotonic map $\phi: \NN^{n} \rightarrow \NN^{n'}$ is said to \emph{change the Stanley depth} by $\ell$ with respect
to $g \in \NN^{n}$ and $g' \in\NN^{n'}$, if it satisfies the following two conditions:
\begin{enumerate}
\item $\phi(g) \leq g'$
\item For each interval $[a',b'] \subset [0,g']$, the (restricted) preimage $\phi^{-1}([a',b']) \cap [0, g]$ can be written as a finite disjoint union $\bigcup_i [a^i, b^i]$ of intervals, such that
\[
\#\set{j\in[n] \with b^i_j=g_j} \geq \#\set{j\in[n'] \with b'_j=g'_j} + \ell \ \ \mbox{~for all }i.
\]
\end{enumerate}
\end{definition}

\begin{notation} 
Let $R = \KK[X_1,\ldots,X_n]$, $R' = \KK[X_1,\ldots,X_{n'}]$.
A map $\phi: \NN^{n} \rightarrow \NN^{n'}$ gives rise to a natural $\KK$-linear map
$\Phi: R \rightarrow R'$, 
which is defined on monomials as
$$
\Phi(X^a):=X^{\phi(a)}.
$$
and extended to $R$ linearly.
Note that this is \emph{not} a ring homomorphism.
\end{notation}

Next proposition justifies the name for a monotonic map \emph{changing the Stanley depth} of the previous definition:

\begin{proposition}\label{prop:sdep}
Let  $n, n' \in \NN$, $R = \KK[X_1,\ldots,X_n]$, $R' = \KK[X_1,\ldots,X_{n'}]$ be two polynomial rings and let $J' \subsetIJ I' \subset R'$ be monomial ideals.
Consider a monotonic map $\phi: \NN^{n} \rightarrow \NN^{n'}$ and set $I := \Phi^{-1}(I')$, $J := \Phi^{-1}(J')$.
Choose $g \in \NN^n$ and $g' \in \NN^{n'}$, such that every minimal generator of $I$ and $J$ divides $X^g$, and every minimal generator of $I'$ and $J'$ divides $X^{g'}$.
%
Let $\ell\in\ZZ$ and assume that $\phi$ changes the Stanley depth by $\ell$ with respect to $g$ and $g'$. Then
\begin{itemize}
	\item[(i)] $I$ and $J$ are monomial ideals, and
	\item[(ii)] $\sdepth\ I/J \geq \sdepth\ I'/J' + \ell$.
\end{itemize}
\end{proposition}
\begin{proof}
	It is clear that $I$ and $J$ are monomial ideals (since monomial ideals correspond to subsets of $\NN^n$ that are closed under ``going up'', and taking the preimage under a monotonic map preserves this property).
	For the second claim, we compute the Stanley depth of $I'/J'$ via an interval partition of $\pijp$ (cf. Theorem \ref{theo:HVZ}).
	Note that the assumption $\phi(g) \leq g'$ together with the equalities $\Phi^{-1}(I') = I$ and $\Phi^{-1}(J') = J$ imply that $\phi^{-1}(\pijp) \cap [0,g] = \pij$.
	Hence, taking the preimages of the intervals in the partition of $\pijp$ yields an interval partition of $\pij$ of the required Stanley depth.
\end{proof}

The following results are useful for constructing maps $\phi$ which satisfy the conditions of the previous proposition:

\begin{lemma}\label{lem:interval}
	Let $n_1, n'_1, n_2, n'_2 \in \NN$.
	For $i=1,2$, let $\phi_i : \NN^{n_i} \rightarrow \NN^{n'_i}$ be  monotonic maps that change the Stanley depth by $\ell_i$ with respect to $g_i \in \NN^{n_i}$ and $g_i'  \in \NN^{n'_i}$.
	Then the product map
	\[ (\phi_1, \phi_2): \NN^{n_1+n_2} \rightarrow \NN^{n'_1+n'_2} \]
	changes the Stanley depth by $\ell_1 + \ell_2$ with respect to $(g_1,g_2)$ and $(g'_1, g_2')$.
\end{lemma}
\begin{proof}
	Let us denote the product map by $\phi := (\phi_1, \phi_2)$. It is enough to consider one interval $[(p_1,p_2), (q_1, q_2)] \subset \NN^{n'_1+n'_2}$.
	By assumption, the preimage $\phi_i^{-1}([p_i,q_i])\bigcap [0,g_i] = \bigcup_j[p^j_{i},q^j_{i}] \subset \NN^{n_i}$ is a disjoint finite union of intervals.
	Then
	\begin{align*}
	\phi^{-1} & ([(p_1,  p_2),  (q_1, q_2)])\cap [0,(g_1,g_2)] = \\
		&= \phi^{-1}(([p_1, q_1] \times \NN^{n_2}) \,\cap\,( \NN^{n_1} \times [p_2, q_2])) \,\cap\, [0,(g_1,g_2)]\\
		&= (\phi_1^{-1}([p_1, q_1]) \times \NN^{n_2}) \,\cap\, (\NN^{n_1} \times \phi_2^{-1}([p_2, q_2])) \cap ([0, g_1] \times \NN^{n_2}) \,\cap\, (\NN^{n_1} \times [0, g_2]) \\
		&= \left( (\phi_1^{-1}([p_1, q_1]) \cap [0,g_1] ) \times \NN^{n_2}\right)\,\cap\, \left(\NN^{n_1} \times (\phi_2^{-1}([p_2, q_2]) \cap [0,g_2])\right) \\
		&= \left( ( \bigcup_j[p^j_{1},q^j_{1}]) \times \NN^{n_2} \right)\,\cap\, \left(\NN^{n_1} \times ( \bigcup_j[p^j_{2},q^j_{2}]) \right) \\
		&= \bigcup_{j,k} [(p^j_{1}, p^k_{2}), (q^j_{1},q^k_{2})],
	\end{align*}
	which implies the statement.
\end{proof}

\begin{remark}\label{rem:id}
The most important special case of this lemma is when one of the maps is the identity. In this case, if $\phi$ is a map changing the Stanley depth by $\ell$, then we can pad it with identities to get a new map $\hat{\phi} = (id, \dotsc, id, \phi)$ that still changes the Stanley depth by $\ell$. In the sequel, $\hat{\phi}$ will always denote the padded version of $\phi$ and $\hat{\Phi}$ the padded version of $\Phi$.
\end{remark}

\begin{lemma}\label{lem:1dim}
	Every monotonic map $\phi: \NN \rightarrow \NN^{n'}$ changes the Stanley depth by $1-n'$ with respect to $g \in \NN$ and $g' := \phi(g)$.
\end{lemma}
\begin{proof}
	Let $Q'=[a',b'] \subset [0,g'] \subset \NN^{n'}$ be an interval and let $Q := \phi^{-1}(Q') \cap [0,g]$ be its (restricted) preimage.
	Let $a$~resp.~$b \in \NN$ be the minimal~resp.~maximal element in $Q$.
	Then $Q \subset [a,b]$ and we claim that we have equality.
	For $c \in[a,b]$, it follows from $\phi(a) \leq \phi(c) \leq \phi(b)$ that $\phi(c) \in [a',b']$, because $\phi(a), \phi(b) \in [a',b']$. Thus $c \in Q$.
	So the preimage of an interval is again an interval.
	It remains to verify the condition (2) of Definition \ref{def:1}, namely
	\[
	\#\set{j\in[1] \with b_j=g_j} \geq \#\set{j\in[n'] \with b'_j=g'_j} + 1 - n' \ \ \mbox{~for all }i.
	\]
	If $b' < g'$, then the right hand side is nonpositive, so the condition is trivially satisfied. On the other hand, if $b' = g'$, then $b = g$ and thus the condition is also satisfied.
\end{proof}

We close this section with a precise description of the behavior of poset maps preserving joins and meets with respect to the property of changing Stanley depth. This result will be not needed in the sequel.

\begin{theorem} 
	Let $n, n' \in \NN$ and let $\phi: \NN^n \rightarrow \NN^{n'}$ be a map that preserves joins and meets.
	Then $\phi$ changes the Stanley depth by $n-n'$  with respect to $g \in \NN^n$ and $g' := \phi(g)$.
\end{theorem}

\begin{proof}
    We assume that $\phi(0) = 0$, as we otherwise replace $\phi$ by $\phi - \phi(0)$ without changing the validity of the statement.
    First we show that $\phi$ is monotonic. Consider $a, b \in \NN^{n}$ with $a \leq b$. Then
    $$
    \phi(a)\le\phi(a) \vee \phi(b) = \phi(a \vee b) = \phi(b),
    $$
    hence $\phi(a) \leq \phi(b)$.

    In view of Lemma \ref{lem:interval} and Lemma \ref{lem:1dim}, it is sufficient to show that $\phi$ is a product of monotonic maps $\phi_i: \NN \rightarrow \NN^{n_i'}$ with $1\leq i \leq n$ and $\sum_i n_i' = n'$.

    For this, let $e_i$ denote the $i$-th unit vector of $\NN^n$.
    Every vector $v =\sum_i \lambda_i e_i\in \NN^n$ can be written as a sum $v = \lambda e_1 +(0, v')$ with $v'\in\NN^{n-1}$.
    As $\phi$ is monotonic, the support of $\phi(\lambda e_1)$ can only increase with $\lambda$. Hence, without loss of generality, we may assume that the maximal support of $\phi(\lambda e_1)$ for $\lambda \gg 0$ is the set of the first $n_1'$ coordinates. That is
    $$
    \phi(\lambda e_1)\in  \NN^{n_1'}\times (0,\ldots, 0).
    $$

    Since $\phi$ preserves meets, for all $\lambda\in \NN$ and $v'\in\NN^{n-1}$ we have
    \[
    0=\phi(0)=\phi(\lambda e_1\wedge (0,v'))=\phi(\lambda e_1)\wedge \phi((0,v')).
    \]
    It follows that for each vector $(0,v')$ with vanishing first coordinate, the support of its image is contained in the last $n' - n_1'$ coordinates. That is
    $$
    \phi((0,v'))\in (0,\ldots, 0)\times \NN^{n' - n_1'}.
    $$
    Then, it holds that
    $$
    \phi(v)=\phi(\lambda e_1 + (0,v')) =\phi(\lambda e_1\vee (0,v'))=\phi(\lambda e_1)\vee \phi((0,v'))=\phi(\lambda e_1) + \phi((0,v')),
    $$
    because the sum equals the join in this case.
    Set $\phi_1:\NN\rightarrow \NN^{n_1'} $, $\phi_1(\lambda)=\pi_{n_1'} \phi(\lambda e_1)$
    where $\pi_{n_1'}$ is the projection on the first $n_1'$ coordinates. Set $\phi_{n-1}:\NN^{n-1}\rightarrow \NN^{n'-n_1'} $, $\phi_{n-1}(v')=\pi_{n'-n_1'} \phi((0,v'))$
    where $\pi_{n'-n_1'}$ is the projection on the last $n'-n_1'$ coordinates.
    Then $\phi$ splits into a direct product
    \[ \phi = (\phi_1, \phi_{n-1}): \NN \times \NN^{n-1} \rightarrow \NN^{n_1'} \times \NN^{n' - n_1'}. \]
    Since $\phi_1$ is a restriction of $\phi$, it follows that  $\phi_1$ is monotonic.

    Iterating this construction yields the desired decomposition of $\phi$.
\end{proof}

\section{Application to Polarizations} \label{section:polarizations}

Let $J\subsetIJ I$ be two monomial ideals in $R$. Then we can choose $e \in \ZZ^n$ bigger than or equal to (in the sense of the order $\leq$)  the join of all generators of $I$ and $J$, and define the polarization $I^p/J^p$ of $I/J$ according to Section \ref{section:pre}.
\medskip

Polarization can also be done step by step. Let $u_1,\ldots, u_m$ be the set of minimal generators of $I$.
Following \cite{H} we define the \emph{$1$-step polarization} of $I$ (with respect to $X_i$) to be the ideal $I^1 \subset R[Y]$
generated by $v_1,\ldots, v_m$, where
\[
	v_j := \begin{cases}
		\frac{Y}{X_i}u_j &\text{ if } X_i^2 \mid u_j \\
		u_j &\text{ otherwise.}
	\end{cases}
\]

If the indeterminate of the polynomial ring with respect which we apply partial polarization is not relevant, we will omit it.

\begin{proposition}\label{Prop:polarization2}
Let $J \subsetIJ I \subset R$ be monomial ideals. Then
\begin{enumerate}
\item $\depth_{R[Y]}\ I^1/J^1=\depth_{R}\ I/J+1$;
\item $H_{I^1/J^1}(t) = \frac{1}{1-t} H_{I/J}(t)$.
\end{enumerate}
\end{proposition}
Here $H_{I/J}(t)$ denotes the Hilbert series with respect to the $\ZZ$-grading.
This proposition is easy and partially known (cf. \cite[p. 47]{H}). We present a complete proof for the reader's convenience.
\begin{proof}
For both claims it suffices to show that $X_n-Y$ is a regular element. By contradiction, assuming $X_n-Y$ is a zero-divisor on $I^1/J^1$, then there exists an ideal $\mathfrak{p} \in \mathrm{Ass}(I^1/J^1)$ with $X_n-Y \in \mathfrak{p}$. Since both $J^1$ and $\mathfrak{p}$ are monomial ideals, there exists a monomial $h \notin J^1$ such that $\mathfrak{p}=\Ann_{R'}h$. Then $h(X_n-Y) \in J^1$, and again $hY \in J^1$, $hX_n\in J^1$, since $J^1$ is monomial. Let $f_1,f_2 $ generators of $J^1$ such that $f_1 | hX_n$ and $f_2 | hY$. If $Y|f_2$ then $X_n | f_2$ and so $X_n | h$. Therefore $X_n^2 | f_1$, then $Y|f_1$. This implies that $Y$ divides $h$, a contradiction.
\end{proof}

One difficulty in proving our main result is that the ideals $I^1$ and $J^1$ do not arise as preimages of a map, but they are rather (generated by) the image of a map. The following Lemma is helpful for constructing Stanley decompositions in this setting.
\begin{lemma}\label{lem:sdep2}
Let $n, n' \in \NN$, $R = \KK[X_1,\ldots,X_{n}]$, $R' = \KK[X_1,\ldots,X_{n'}]$ be two polynomial rings and let $J \subsetIJ I \subset R$ be monomial ideals.
Let $\phi: \NN^{n} \rightarrow \NN^{n'}$ be a map and let $J' \subsetIJ I'  \subset R'$ be the ideals generated by $\Phi(I)$ resp. $\Phi(J)$.
Let $\Omega$ be a finite set and let $I/J = \bigoplus_{i \in \Omega} X^{a_i} \KK[Z_i]$ be a Stanley decomposition of $I/J$.
Let $Z'_i, i\in \Omega$ be a collection of subsets of $\set{X_1,\ldots,X_{n'}}$.

Assume that $\phi$ is injective, monotonic, and preserves joins.
Assume moreover, that
\[
	\Phi(X^{a_i}) \KK[Z'_i] \cap \Phi(R) = \Phi( X^{a_i} \KK[Z_i] )\quad \quad \quad (\star)
\]
for each $i\in \Omega$.

Set $V=\sum_{i \in \Omega} \Phi(X^{a_i}) \KK[Z'_i]$ as a graded vector space. Then
\[
V=\bigoplus_{i \in \Omega} \Phi(X^{a_i}) \KK[Z'_i]
\]
is a direct sum and it holds that $V \subset I' / J'$.
\end{lemma}
\begin{proof}
First, we show that the sum is direct.
On the contrary, suppose that there are indices $i \neq j$ such that a monomial appears in both parts. This monomial is then a common multiple of $\Phi(X^{a_i})$ and $\Phi(X^{a_j})$.
Then also the least common multiple of $\Phi(X^{a_i})$ and $\Phi(X^{a_j})$ appears in both parts.
As $\Phi$ preserves the least common multiple of two monomials, it follows that the least common multiple of $\Phi(X^{a_i})$ and $\Phi(X^{a_j})$ is $\Phi(X^{a_i\vee a_j})$ (where $X^{a_i\vee a_j}$ is the least common multiple of $X^{a_i}$ and $X^{a_j}$). But now, the condition $(\star)$ and the fact that $\phi$ is injective imply that $X^{a_i\vee a_j} \in X^{a_i} \KK[Z_i] \cap X^{a_j} \KK[Z_j]$, a contradiction.

Next, we show that $V \subset I'/J'$.
Every $X^{a_i}$ is contained in $I$, hence $V$ is contained in $I'$.
It remains to show that no monomial in $\Phi(X^{a_i}) \KK[Z'_i]$ is contained in $J'$.
Assume on the contrary that such a monomial exists.
Then it is a common multiple of $\Phi(X^{a_i})$ and $\Phi(X^a)$ for a minimal generator of $X^a$ of $J$.
Again, it follows that the least common multiple of $\Phi(X^{a_i})$ and $\Phi(X^a)$ is contained in $\Phi(X^{a_i}) \KK[Z'_i] \cap J' \cap \Phi(R)$.
But then the least common multiple of $X^{a_i}$ and $X^a$ is contained in $X^{a_i} \KK[Z_i] \cap J$, a contradiction.
\end{proof}

\begin{theorem}\label{Theorem:main}
	Let $J \subsetIJ I \subset R$ be two monomial ideals, and let $J^1 \subsetIJ I^1 \subset R[Y]$ be their $1$-step polarizations. Then
	\[
	\sdepth\ I/J = \sdepth\ I^1/J^1 - 1.
	\]
\end{theorem}

\begin{proof}
	Without loss of generality, we assume that $X^2_n$ divides one of the generators of $I$ or $J$ and we apply polarization with respect to $X_n$.
	Consider the map $\phi: \NN \rightarrow \NN^2$ defined by
	\[
		\phi(i) := \begin{cases}
			(i-1,1) &\text{ if } i \geq 2\\
			(i,0) &\text{ if } i =0,1.
		\end{cases}
	\]
\begin{figure}[h]
\begin{tikzpicture}[scale=0.7, radius=0.14]
\newcommand{\Anzx}{5}
\newcommand{\Anzy}{3}
	\draw[gray!50] ($(-0.5,-0.5)$) grid ($(\Anzx+0.5,\Anzy+0.5)$);
	\draw[->] (-1,0) -- ($(\Anzx+1,0)$);
	\draw[->] (0,-1) -- (0,4);

	\foreach \x in {0,1}
		\fill (\x,0) circle;

	\foreach \x in {1,..., \Anzx}
		\fill (\x,1) circle;

\end{tikzpicture}
\caption{The image of $\phi$ from the proof of Theorem \ref{Theorem:main}.}
\end{figure}
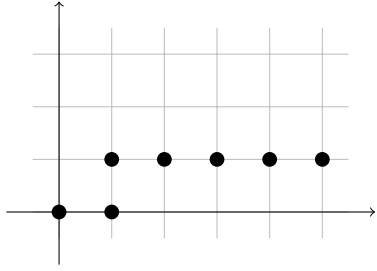

	It is easy to see that $\phi$ is injective, monotonic and preserves joins.
	Further, observe that $\phi(a) \leq \phi(b)$ if and only if $a \leq b$ for $a,b \in \NN$.
	Moreover, all these properties carry over to $\hat{\phi}:\NN^{n}\rightarrow\NN^{n+1}$ and $\hat{\Phi}:R\rightarrow R[Y]$ (see also Remark \ref{rem:id}).
	
	By definition, it holds that $I^1$ is generated by $\hat{\Phi}(I)$, hence $\hat{\Phi}^{-1}(I^1) \supseteq I$.
	We claim that $\hat{\Phi}^{-1}(I^1) = I$.
	To see this, consider a monomial $X^b\in\hat{\Phi}^{-1}(I^1)$.
	Then $\hat{\Phi}(X^b)$ is a multiple of a generator $\Phi(X^a)$ of $I^1$, where $X^a \in I$. By the observation above, $X^b$ is then a multiple of $X^a$ and thus contained in $I$.
	Analogously it holds that $\hat{\Phi}^{-1}(J^1) = J$.
	
	Let $g=(g_1, \ldots, g_n)$ be the join of the exponents of the monomial generators of $I$ and $J$ (we assume $g_n\ge 2$).
	Then $g' = (g_1,\ldots,g_{n-1},g_n - 1, 1) = \hat{\phi}(g)$ is the join of the exponents of the monomial generators of $I^1$ and $J^1$.
	It follows from Lemma \ref{lem:1dim} that $\hat{\phi}$ changes the Stanley depth by $-1$ with respect to $g_n$ and $(g_n-1,1)$.
	Hence by Lemma \ref{lem:interval} and Proposition \ref{prop:sdep} it holds that $\sdepth\ I/J \geq \sdepth\ I^1/J^1 - 1$.

\newcommand{\hphi}[1]{\hat{\Phi}(#1)}
\newcommand{\nt}{\tilde{N}}
\newcommand{\mt}{\tilde{M}}
	Now we turn to the second inequality.
	Let us consider a Stanley decomposition
	$I/J = \bigoplus_{i} X^{a_i} \KK[Z_i]$
	of $I/J$.
	Remark that $\hphi{X_n X^a}/ \hphi{X^a} \in \set{X_n, Y}$ for every monomial $X^a \in R$, which allows us to define
	\[  Z'_i := \begin{cases}
		Z_i \cup \set{Y} & \text{~if~}~ X_n \in Z_i\\
		Z_i \cup \left(\set{X_n, Y} \setminus \set{\frac{\hphi{X_n X^{a_i}}}{\hphi{X^{a_i}}}}\right)& \text{~otherwise.}
	\end{cases} \]

	We claim that
	\[
	V := \bigoplus_{i} \hat{\Phi}(X^{a_i}) \KK[Z'_i] \quad \quad \quad (\star\star)
	\]
	is a Stanley decomposition of $I^1/J^1$.
	First, we show that
	\begin{equation*} 
		\hat{\Phi}(X^{a_i}) \KK[Z'_i] \cap \hat{\Phi}(R) = \hat{\Phi}( X^{a_i} \KK[Z_i])
	\end{equation*}
	for each $i$. The inclusion ``$\supseteq$'' is clear.

	For the other inclusion, let $M':= \hat{\Phi}(X^{a_i}) N'$ be a monomial in the left-hand side and let $M := \hat{\Phi}^{-1}(M') \in R$ be its preimage. Note that $\hphi{X^{a_i}} \mid M'$ implies that $X^{a_i} \mid M$ (it follows from the observation above that  $\phi(a) \leq \phi(b)$ if and only if $a \leq b$ for $a,b \in \NN$), hence we may define $N := M / X^{a_i}$. It suffices to show that $N \in \KK[Z_i]$, because then $M' = \hphi{X^{a_i} N}$ is contained in the right-hand side.
	
	By definition, the map $\hat{\Phi}$ may change only the exponents of $X_n$ and $Y$. Therefore the equality $\hphi{X^{a_i}N} = \hphi{X^{a_i}}N'$ implies that $N'$ and $N$ may differ only by multiplication and/or division by (powers of) $X_n$ and $Y$.
	Hence every other variable appearing in $N$ also appears in $N' \in \KK[Z_i']$, so it is contained in $Z_i'$ and (thus) in $Z_i$.
	Furthermore, $N \in R$ implies that $Y \nmid N$. So we only need to prove the following: If $X_n \mid N$, then $X_n \in Z_i$.

	So assume that $X_n \mid N$. Then $X_n X^{a_i}\mid N X^{a_i}$ and thus $\hphi{X_n X^{a_i}}\mid \hphi{N X^{a_i}} = M'$.
This further implies that
\[
\frac{\hphi{X_n X^{a_i}}}{\hphi{X^{a_i}}} \ \vrule\  \frac{M'}{\hphi{X^{a_i}}} = N'. 
\]
	As $N' \in \KK[Z_i']$, it now follows from the definition of $Z_i'$ that $X_n \in Z_i$.

	It is easy to see that $\hat{\Phi}$ satisfies the assumptions of Lemma \ref{lem:sdep2}, so we conclude that the sum in $ (\star \star) $ is direct and that $V \subset I^1 / J^1$. It remains to show that $V = I^1/J^1$.
	For this, we compute the graded Hilbert series of $V$:
	\begin{align*}
	 H_V(t) &= \frac{\sum_i t^{|a_i|} (1-t)^{n+1-|Z'_i|}}{(1-t)^{n+1}} \\
	 	&= \frac{\sum_i t^{|a_i|} (1-t)^{n-|Z_i|}}{(1-t)^{n+1}} \\
	 	&= \frac{1}{1-t} H_{I/J}(t) = H_{I^1/J^1}(t).
	\end{align*}
	Here, $|a_i|$ denotes the sum of the components of $a_i$.
	In the first and third equality we used the decompositions of $V$ resp. $I/J$ given above.
	For the last equality we used Proposition \ref{Prop:polarization2}.
	As we have already shown that $V \subset I^1/J^1$, the claim follows.
	
	We conclude that the sum in $ (\star \star) $ is a Stanley decomposition of $I^1/J^1$. As $|Z'_i| = |Z_i| + 1$ for each $i$, it follows that $\sdepth\ I^1/J^1 \geq \sdepth\ I/J + 1$.
\end{proof}

\noindent Iteration of Theorem \ref{Theorem:main} and Proposition \ref{Prop:polarization2} has one immediate consequence:
\begin{corollary}\label{cor:conj}
Let $J \subsetIJ I \subset R$ be monomial ideals, and let $I^p,J^p \subset R^p$ be their polarizations.
Then
\[
\sdepth\ I/J - \depth\ I/J = \sdepth\ I^p/J^p - \depth\ I^p/J^p.
\]
In particular, \cite[Conjecture 62]{H} is true.
\end{corollary}

Note that the preceding corollary effectively reduces Stanley's conjecture to the squarefree case:
\begin{corollary}\label{cor:polarization}
	Let $J \subsetIJ I \subset R$ be monomial ideals, and let $I^p,J^p \subset R^p$ be their polarizations.
	Then $I/J$ satisfies the Stanley Conjecture if and only if $I^p/J^p$ satisfies it too.
\end{corollary}

\begin{remark}
\begin{enumerate}
	\item We would like to point out that the ``only if''-part of Corollary \ref{cor:polarization} already appeared in \cite[Theorem 3.5]{A} in the quite particular case $J=(0)$ and $R/I$ Cohen-Macaulay. We obtain in contrast a full reduction of the Stanley's conjecture to the squarefree case.
	\item An invariant related to the Stanley depth is the \emph{Hilbert depth}, which was introduced by Uliczka \cite{U} in the standard $\ZZ$-graded case and by Bruns, Krattenthaler and Uliczka \cite{BKU} in the multigraded case.
	Theorem \ref{Theorem:main} and \Cref{cor:conj} hold \emph{mutatis mutandis} for the Hilbert depth:
	since the multigraded Hilbert depth coincides with the Stanley depth in our situation (cf. \cite[Proposition 2.8]{BKU}), we have
	\[
		\hdepth_{R[Y]}\ I^1/J^1=\hdepth_R\ I/J+1.
	\]
	In the $\ZZ$-graded case, the equality
	\[
		\hdepth_{R[Y]}^1\ I^1/J^1=\hdepth_R^1\ I/J+1
	\]
	is easily deduced from \cite[Theorem 3.2]{U}.
\end{enumerate}
\end{remark}

It was previously known (see \cite{HJY}) that the Stanley conjecture for algebras of the type $R/I$ (where $I\subset R$ is a monomial ideal) implies the conjecture (also due to Stanley \cite[Conjecture 2.7]{St2}) that every Cohen-Macaulay simplicial complex is partitionable. We show that they are in fact equivalent:

\begin{corollary}\label{cor:equ}
Let $I\subset R$ be a monomial ideal. Stanley's conjecture for algebras of the type $R/I$ is equivalent to Stanley's conjecture that every Cohen-Macaulay simplicial complex is partitionable.
\end{corollary}

\begin{proof} It is known that Stanley's conjecture holds for all algebras of the type $R/I$ if and only if it holds for all Cohen-Macaulay such algebras, cf. \cite[Corollary 3.2]{HJZ}. As it was pointed out in \cite{H}, this together with Corollary \ref{cor:polarization} completes the proof.
\end{proof}

\begin{corollary}
	Stanley's conjecture holds for quotients $R/I$ of Gorenstein monomial ideals $I \subset R$ with at most $8$ generators.
\end{corollary}
\begin{proof}
	Polarization preserves the Gorenstein property and the number of generators, so by Corollary \ref{cor:polarization} we may assume that $I$ is squarefree.
	Moreover, by \cite[Prop. 5.1]{IJ} we may assume that every variable of $R$ appears in a generator of $I$. So $I$ is a Stanley-Reisner ideal of a homology sphere. But these homology spheres have been classified in \cite{K}. In particular, they are all polytopal and thus shellable. Since the Stanley conjecture is known to hold in this case by \cite{HP} or \cite[p.~15]{H}, the claim follows.
\end{proof}

\section{Further Applications to Stanley decompositions} \label{section:furtherapplications}

Let $I, J \subset R$ be monomial ideals with $J \subsetIJ I$. First of all, the techniques introduced in Section \ref{section:posetmaps}~ allows us to generalize results of
Cimpoea\c{s} \cite[Lemma 1.1]{Ci2}, and Ishaq and Qureshi \cite[Lemma 2.1]{MM} :
\begin{proposition}\label{prop:unsus1}
	Let $k\in\NN$.
	Let $I'$ and $J'$ be the monomial ideals obtained from $I$ and $J$ in the following way:
	Each generator whose degree in $X_n$ is at least $k$ is multiplied by $X_n$, and all other generators are taken unchanged.
	Then
	\[
	\sdepth\ I/J = \sdepth\ I' / J'.
	\]
\end{proposition}
\begin{proof}
	Consider the maps $\phi, \psi: \NN \rightarrow \NN$ defined by
	\[
	\phi(i) := \left\{\begin{aligned}
		&i & \text{ if } i < k, \\
		&i+1 & \text{ if } i \geq k;
	\end{aligned}\right.
	\qquad
	\psi(i) := \left\{\begin{aligned}
		&i & \text{ if } i \leq k, \\
		&i-1 & \text{ if } i > k.
	\end{aligned}\right.
	\]
	Note that both maps change the Stanley depth by $0$ with respect to $g \in \NN$ and $g+1 \in \NN$ or vice versa, for each $g \geq k$.
	Moreover, by defining the maps $\hat{\Phi}, \hat{\Psi}$ as in Remark \ref{rem:id}, we have $\hat{\Phi}^{-1}(I') = I$, $\hat{\Psi}^{-1}(I) = I'$, and similarly for $J$.
	So the claim follows from Proposition \ref{prop:sdep}.
\end{proof}

The following result generalizes \cite[Lemma 2.3]{S}.
\begin{proposition}\label{prop:unsus2}
	Let $I'$ and $J'$ be the monomial ideals in $\KK[X_1,\ldots, X_n,X_{n+1}]$ obtained from $I$ and $J$ in the following way:
	In each minimal generator of $I$ and $J$, every occurrence of the variable $X_n$ is replaced with the product $X_n X_{n+1}$.
	Then
	\[
	\sdepth\ I' / J' = \sdepth\ I / J + 1.
	\]
\end{proposition}
\begin{proof}
	Consider the map $\phi: \NN^2 \rightarrow \NN$, $\phi(a,b) := \min(a,b)$.
	Let $g$ be the maximal degree of $X_n$ in a minimal generator of $J$ or $I$.
	Then
	\[
	\phi^{-1}([a,b]) \cap [(0,0),(g,g)] = [(a,a), (b,g)] \dot\cup [(b+1, a), (g, b)]
	\]
	for $0 \leq a \leq b \leq g$, hence $\phi$ changes the Stanley depth by $1$ with respect to $(g,g)$ and $g$ (remark that the union is disjoint).
	By defining the map $\hat{\Phi}$ as in Remark \ref{rem:id}, we have $\hat{\Phi}^{-1}(I) = I'$ and similarly for the ideal $J$.
	Therefore, Proposition \ref{prop:sdep} yields the inequality ``$\geq$''.
	The reverse inequality ``$\leq$'' is a consequence of \cite[Proposition 5.2]{IJ}.
	Alternatively, the reverse inequality follows from Lemma \ref{lem:1dim}, applied to the diagonal map $\NN \rightarrow \NN^2, a \mapsto (a,a)$.
\end{proof}

\begin{remark}
	\begin{enumerate}
	\item Let us remark that the two propositions in this section are enough to reduce the computation of a monomial complete intersection to the case of the maximal ideal. We follow the line of reasoning of Shen \cite[Theorem 2.4]{S}. A monomial complete intersection ideal $I$ is generated by monomials with pairwise disjoint support. So by Proposition \ref{prop:unsus1}, its Stanley depth does not change if we replace $I$ with its radical. Then Proposition \ref{prop:unsus2} allows us to replace $I$ by an ideal generated by variables.
	\item Note that if $I$ is a Stanley-Reisner ideal, then the process described in Proposition \ref{prop:unsus2} corresponds to the one-point suspension \cite{JL}, so topologically it is a suspension. This gives a geometric explanation why the depth increases exactly by $1$.
	\end{enumerate}
\end{remark}

\section{Acknowledgements}

The authors would like to thank  Dorin Popescu and Yi-Huang Shen for their suggestions and useful comments.

\bibliographystyle{alpha}
\bibliography{LCM}

\end{document}